\documentclass[12pt]{article}
\usepackage[utf8]{inputenc}
\usepackage{amsthm}
\usepackage{amsfonts}
\usepackage{amsmath}
\usepackage{amssymb}
\usepackage{epsfig}
\usepackage{url}
\newtheorem{theorem}{Theorem}
\newtheorem{definition}{Definition}

\newtheorem{lemma}[theorem]{Lemma}
\newtheorem{corollary}[theorem]{Corollary}
\newtheorem{observation}[theorem]{Observation}
\newcommand{\ch}{\text{ch}}

\newenvironment{subproof}{%
  \begin{proof}[Subproof]%
}{%
  \end{proof}%
}

\title{Correspondence coloring and its application to list-coloring planar graphs without cycles of lengths $4$ to $8$}

\author{
Zdeněk Dvořák\thanks{Charles University, Prague, Czech Republic. E-mail: {\tt rakdver@iuuk.mff.cuni.cz}.  Supported by project GA14-19503S (Graph coloring and structure) of Czech Science Foundation.}\and
Luke Postle\thanks{University of Waterloo. E-mail: {\tt lpostle@uwaterloo.ca}.  Partially supported by NSERC under Discovery Grant No. 2014-06162.}}
\date{}

\begin{document}
\maketitle

\begin{abstract}
We introduce a new variant of graph coloring called \emph{correspondence coloring} which generalizes list coloring and
allows for reductions previously only possible for ordinary coloring.  Using this tool, we
prove that excluding cycles of lengths $4$ to $8$ is sufficient to guarantee $3$-choosability of a planar graph, thus answering a question of Borodin.
\end{abstract}

\section{Introduction}

The study of colorings of planar graphs has a long history, starting with the Four Color Problem in the 19th century.
Let us recall that a \emph{$k$-coloring} of a graph $G$ is a function that assigns one of $k$ colors to each of
the vertices of $G$, so that adjacent vertices have different colors, and we say that $G$ is \emph{$k$-colorable} if it has a $k$-coloring.
The Four Color Problem (that is, whether every planar graph is $4$-colorable) was eventually resolved in the affirmative by Appel and Haken~\cite{AppHak1} in 1976, but over the time
it has inspired many other coloring results; for example, a significantly easier proof that planar graphs are $5$-colorable
dates back to the works of Kempe and Heawood, and Gr\"otzsch~\cite{grotzsch1959} proved that planar triangle-free graphs
are $3$-colorable.

It also inspired the study of many variants of graph coloring, most prominently \emph{list coloring}.
A \emph{list assignment} for a graph $G$ is a function $L$
that to each vertex $v\in V(G)$ assigns a list $L(v)$ of colors. An \emph{$L$-coloring} of $G$ is a proper coloring $\varphi$
such that $\varphi(v)\in L(v)$ for every $v\in V(G)$.  A graph $G$ is \emph{$k$-choosable} if there exists an $L$-coloring of $G$
for every assignment $L$ of lists of size $k$ to vertices of $G$.

Clearly, every $k$-choosable graph is $k$-colorable, but the converse is known not to be true.  For example,
while every planar triangle-free graph is $3$-colorable, there exist such graphs that are not $3$-choosable~\cite{voigt1995},
and while every planar graph is $4$-colorable, not all are $4$-choosable~\cite{voigt1993}.  On the other hand,
planar graphs are $5$-choosable~\cite{thomassen1994}, and every planar graph without cycles of lengths $3$ and $4$ is
$3$-choosable~\cite{thomassen1995-34}.

The method of reducible configurations is a broadly applicable technique used to prove both colorability and list-colorability results,
especially in sparse graphs: to show that a graph $G$
is colorable, we find a special subgraph $H$ of $G$ such that after modifying $G$ locally in
the neighborhood of $H$ (performing a \emph{reduction}), we obtain a smaller graph $G'$ such that every
coloring of $G'$ can be extended to a coloring of $G$.  Repeating this procedure, we eventually reduce the graph to one
of bounded size, whose colorability can be shown directly.

It should not come as a surprise that many of the reductions used for $k$-colorability are not possible in the list-coloring setting.
For ordinary coloring, the reductions commonly involve identifications of vertices.  For example, to prove that every planar graph
$G$ is $5$-colorable, we first observe that $G$ contains a vertex $v$ of degree at most $5$.  If $\deg(v)\le 4$, then every $5$-coloring
of $G-v$ extends to a $5$-coloring of $G$.  If $\deg(v)=5$, then we find two non-adjacent neighbors $x$ and $y$ of $v$, and let
$G'$ be the graph obtained from $G-v$ by identifying $x$ and $y$ to a new vertex $w$.  Given a $5$-coloring $\varphi$ of $G'$,
we then obtain a $5$-coloring of $G$ by giving both $x$ and $y$ the color $\varphi(w)$ and choosing a color for $v$ distinct
from the colors of its neighbors.

In contrast, for list coloring it is in general not possible to identify vertices, since they might have different lists.
Hence, this type of argument for ordinary coloring does not translate directly to the list coloring setting
(and e.g., no proof of $5$-choosability of planar graphs based on reducible configurations is known).
To overcome this difficulty, in this paper we devise a generalization of list coloring which we call \emph{correspondence coloring}
and which enables reductions using vertex identification.  We should note that the use of vertex identification
for correspondence coloring is somewhat limited, e.g., we cannot emulate the example in the previous paragraph as we are essentially
restricted to identifications which do not create parallel edges. However, we present a non-trivial application
in proving the following result, answering a question raised by Borodin~\cite{borsurvey}.

\begin{theorem}\label{thm-main}
Every planar graph $G$ without cycles of lengths $4$ to $8$ is $3$-choosable.
\end{theorem}

\subsection{Correspondence coloring}

Let us now give the definition and some basic properties of correspondence coloring.
Before that, let us point out that all graphs considered in this paper are simple, without loops or parallel edges;
while this bears little consequence for ordinary and list coloring, it is important in the correspondence coloring
setting.

We are again coloring the vertices of a graph $G$ (either using a fixed set of colors
or from lists assigned to the vertices; as we will see momentarily, this
distinction does not make much difference for correspondence coloring).
However, coloring a vertex $u$ by a color $c$ may prevent different colors from
being used at its neighbors (for each edge $uv$, we are prescribed a
\emph{correspondence} which matches some of the colors at $u$ with some of the colors at $v$).  More formally:

\begin{definition}
Let $G$ be a graph.
\begin{itemize}
\item A \emph{correspondence assignment} for $G$ consists of a list assignment $L$ and
a function $C$ that to every edge $e=uv\in E(G)$
assigns a partial matching $C_e$ between $\{u\}\times L(u)$ and $\{v\}\times L(v)$ (the Cartesian product
is used to distinguish the vertices of $C_e$ in case the same color appears both in $L(u)$ and $L(v)$).
\item An \emph{$(L,C)$-coloring} of $G$ is a function $\varphi$ that to each vertex $v\in V(G)$ assigns a color $\varphi(v)\in L(v)$,
such that for every $e=uv\in E(G)$, the vertices $(u,\varphi(u))$ and $(v,\varphi(v))$ are non-adjacent in $C_e$.
We say that $G$ is $(L,C)$-colorable if such an $(L,C)$-coloring exists.
\end{itemize}
\end{definition}
Clearly, this generalizes list coloring, as when we set $C_{uv}$ to match exactly the common colors of $L(u)$ and $L(v)$
for each $uv\in E(G)$, an $(L,C)$-coloring is an $L$-coloring.

Since we specify for each color to which colors at adjacent vertices it corresponds, we can rename the colors arbitrarily
(while updating the correspondence assignment) without affecting the existence of the correspondence coloring.
More precisely, if $(L,C)$ is a correspondence assignment for $G$, $c_1\in L(v)$, and $c_2\not\in L(v)$,
then consider the correspondence assignment $(L',C')$ such that $L'(u)=L(u)$ for $u\in V(G)\setminus\{v\}$,
$L'(v)=(L(v)\cup \{c_2\})\setminus \{c_1\}$, and $C'_e$ is obtained from $C_e$ by replacing the vertex $(v,c_1)$
by $(v,c_2)$ for each $e\in E(G)$.  Then every $(L,C)$-coloring $\varphi$ of $G$ can be transformed into an $(L',C')$-coloring
of $G$ by changing the color of $v$ to $c_2$ if $\varphi(v)=c_1$, and vice versa.  We say that $(L',C')$ is obtained
from $(L,C)$ by \emph{renaming (at vertex $v$)}. Two correspondence
assignments are \emph{equivalent} if one can be obtained from the other by repeated renaming.  The following
facts follow from the previous analysis.

\begin{observation}\label{obs-equiv}
Let $(L,C)$ and $(L',C')$ be equivalent correspondence assignments for a graph $G$,
such that $(L',C')$ is obtained from $(L,C)$ by a sequence of renamings at vertices belonging
to a set $X\subseteq V(G)$.  If $\varphi$ is an $(L,C)$-coloring of $G$, then there exists
an $(L',C')$ coloring $\varphi'$ of $G$ such that $\varphi'(v)=\varphi(v)$ for all $v\in V(G)\setminus X$.
In particular, $G$ is $(L,C)$-colorable if and only if $G'$ is $(L',C')$-colorable.
\end{observation}

For an integer $k$, let $[k]$ denote the set $\{1,\ldots, k\}$.

\begin{observation}\label{obs-kass}
Let $(L,C)$ be a correspondence assignment for a graph $G$, such that $L$ assigns lists
of the same size $k$ to each vertex of $G$.  Then there exists an equivalent correspondence
assignment $(L',C')$ such that $L'(v)=[k]$ for all $v\in V(G)$.
\end{observation}

That is, as we alluded to before, the lists do not matter for correspondence coloring (as long as all
vertices use the same number of colors).  The fact that we can make all the lists be the same
is crucial for our application of correspondence coloring, as we now can make sense of vertex identification.
These observations motivate the following definition.

\begin{definition}
Let $G$ be a graph.
\begin{itemize}
\item A \emph{$k$-correspondence assignment} for $G$ is a function $C$ that to each edge $e=uv\in E(G)$
assigns a partial matching $C_e$ between $\{u\}\times [k]$ and $\{v\}\times [k]$.
Letting $L$ be the correspondence assignment that to each vertex assigns the list $[k]$,
we say that each $(L,C)$-coloring of $G$ is a \emph{$C$-coloring} of $G$, and we say that
$G$ is \emph{$C$-colorable} if such a $C$-coloring exists.
\item The \emph{correspondence chromatic number} of $G$ is the smallest integer $k$ such that $G$ is $C$-colorable
for every $k$-correspondence assignment $C$.
\end{itemize}
\end{definition}

By Observations~\ref{obs-equiv} and \ref{obs-kass}, we see that every graph with correspondence
chromatic number $k$ is $k$-choosable.

To make the relationship between the correspondence chromatic number of a graph and its choosability clearer, 
we need to introduce the notion of consistency of a correspondence assignment.
Let $(L,C)$ be a correspondence assignment for a graph $G$, and let $W=v_1v_2\ldots v_m$ with $v_m=v_1$
be a closed walk of length $m-1$ in $G$.  We say that the assignment $(L,C)$ is \emph{inconsistent on $W$} if
there exist colors $c_1$, \ldots, $c_m$ such that $c_i\in L(v_i)$ for $i=1,\ldots, m$, $(v_i,c_i)(v_{i+1},c_{i+1})$
is an edge of $C_{v_iv_{i+1}}$ for $i=1,\ldots, m-1$, and $c_1\neq c_m$.  Otherwise, $(L,C)$ is \emph{consistent on $W$}.
We say that a correspondence assignment $(L,C)$ is \emph{consistent} if $(L,C)$ is consistent on every closed walk in $G$.
We omit the mention of the fixed list assignment $L$ in the case of $k$-correspondence assignments.

\begin{figure}
\begin{center}
\includegraphics{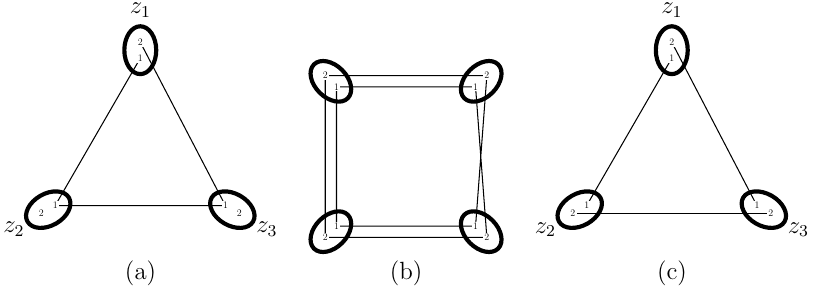}
\end{center}
\caption{Correspondence assignment examples}\label{fig-examples}
\end{figure}

For example, suppose that $G$ is the triangle $z_1z_2z_3$ and $C$ is the $2$-correspondence assignment consisting of the
edges $(z_1,1)(z_2,1)$, $(z_2,1)(z_3,1)$, and $(z_3,1)(z_1,2)$, see Figure~\ref{fig-examples}(a).  Then $C$ is inconsistent on $z_1z_2z_3z_1$ as
witnessed by colors $1,1,1,2$, but $C$ is consistent on $z_3z_1z_2z_3$.

Clearly, the notion of consistency is invariant under renaming of colors.
\begin{observation}\label{obs-consist}
Let $(L,C)$ and $(L',C')$ be equivalent correspondence assignments for a graph $G$.  For every closed walk $W$
in $G$, the correspondence assignment $(L,C)$ is consistent on $W$ if and only if $(L',C')$ is.
\end{observation}

We will now show that list coloring is exactly correspondence coloring with consistent assignments.

\begin{lemma}\label{lemma-chooscon}
A graph $G$ is $k$-choosable if and only if $G$ is $C$-colorable for every consistent $k$-correspondence assignment $C$.
\end{lemma}
\begin{proof}
Suppose first that $G$ is $C$-colorable for every consistent $k$-correspondence assignment $C$,
and let $L$ be an assignment of lists of size $k$ to vertices of $G$.  Let $(L,C')$ be the correspondence
assignment such that for each $uv\in E(G)$, $C'_{uv}$ contains exactly the edges $(u,c)(v,c)$ for all $c\in L(u)\cap L(v)$.
Clearly, $(L,C')$ is consistent, and by Observations~\ref{obs-kass} and \ref{obs-consist}, it is equivalent
to a consistent $k$-correspondence assignment $C$.  By assumption, $G$ is $C$-colorable, and by Observation~\ref{obs-equiv},
it is $(L,C')$-colorable; equivalently, $G$ is $L$-colorable. 
This shows that $G$ is $k$-choosable.

Suppose now that $G$ is $k$-choosable, and let $C$ be a consistent $k$-correspondence assignment for $G$.
Let $H$ be the graph with vertex set $V(G)\times [k]$ and edge set $\bigcup_{e\in E(G)} E(C_e)$.
Note that the consistency of $C$ is equivalent to the fact that for every $v\in V(G)$, each component
of $H$ intersects $\{v\}\times [k]$ in at most one vertex.  For each $v\in V(G)$, let $L(v)$ be the set of
components of $H$ that intersect $\{v\}\times [k]$; clearly $|L(v)|=k$.  By assumption,
there exists an $L$-coloring $\varphi$ of $G$, which we can transform into a $C$-coloring by giving each vertex
$v\in V(G)$ the unique color $c\in[k]$ such that $(v,c)\in V(\varphi(v))$.  Consequently, $G$ is $C$-colorable.
\end{proof}

Let us remark that choosability and correspondence coloring with inconsistent assignments may differ; for example,
even cycles are $2$-choosable, but they have correspondence chromatic number $3$ (see Figure~\ref{fig-examples}(b)
giving a $2$-correspondence assignment for that a 4-cycle cannot be colored).  Another clear distinction is
that while parallel edges do not affect choosability, they matter for correspondence coloring, as correspondences
on parallel edges may differ.

\subsection{Coloring planar graphs with cycles of forbidden lengths}

As we already mentioned, every planar triangle-free graph is $3$-colorable, and
forbidding cycles of lengths $3$ and $4$ is sufficient to guarantee
$3$-choosability.  These facts motivated a plethora of further results and
conjectures concerning sufficient conditions guaranteeing $3$-colorability of a
planar graph, see a survey of Borodin~\cite{borsurvey} for more information.

One of the most influential problems was proposed by Steinberg~\cite{conj-stein}, who conjectured that every planar graph without cycles of lengths $4$ and $5$
is $3$-colorable.  This conjecture was recently disproved by Cohen-Addad et al.~\cite{CohenAddad2016}.  On the positive side,
Borodin et al.~\cite{bor47} proved that every planar graph without cycles of lengths $4$ to $7$ is $3$-colorable,
and minor improvements over this result were obtained later; e.g., Borodin et al.~\cite{col457}
show that it suffices to exclude $5$-cycles, $7$-cycles, and $4$-cycles with a chord.
This leaves open the question of whether it suffices to forbid cycles of lengths $4$ to $6$,
which seems hard to resolve.

Concerning the list-coloring version of Steinberg's problem, the direct generalization of Steinberg's conjecture
was disproved already by Voigt~\cite{Voigt45}, who constructed a planar graph without cycles of lengths $4$ and $5$ that is not $3$-choosable.
On the positive side, the proof of Borodin~\cite{choos49} that every planar graph without cycles of lengths $4$ to $9$ is $3$-colorable
also implies that such graphs are $3$-choosable, and in his survey~\cite{borsurvey}, Borodin points out that it has been open for more than 15
years whether every planar graph without cycles of lengths $4$ to $8$ is $3$-choosable.  In Theorem~\ref{thm-main},
we prove this to be the case.  Our method is to prove a stronger result in the setting of correspondence coloring,
which implies Theorem~\ref{thm-main} by Lemma~\ref{lemma-chooscon}.

\begin{theorem}\label{thm-maingen1}
Every planar graph $G$ without cycles of lengths $4$ to $8$ is $C$-colorable for every
$3$-correspondence assignment $C$ that is consistent on every closed walk of length $3$ in $G$.
\end{theorem}

It might seem more natural to prove that planar graphs without cycles of lengths $4$ to $8$ have
correspondence chromatic number at most $3$ (i.e., not to restrict ourselves to correspondence assignments
that are consistent on closed walks of length $3$), and indeed, we do not know any counterexample to this
stronger claim.  However, we need to assume the consistency of closed walks of length $3$ in order to
perform the main reduction (Lemma~\ref{lemma-tetrad}) of our proof, and Theorem~\ref{thm-maingen} is still
stronger than Theorem~\ref{thm-main} by Lemma~\ref{lemma-chooscon}. 

It is not known whether excluding cycles of lengths $4$ to $7$, or even $4$ to
$6$, forces a planar graph to be $3$-choosable.  While the former question
might be resolved with some extra effort using the methods of this paper, the
latter one is wide open even for ordinary coloring.

\subsection{More results on correspondence coloring}

Proving Theorem~\ref{thm-maingen} shows the utility of the notion of correspondence coloring
as a tool to study list colorings.  Following our announcement of this result and
its usage of correspondence coloring, several other applications of correspondence coloring
were found.  In~\cite{corrprob}, the possibility to introduce new correspondences was used to keep
probabilities uniform across multiple iterations of a probabilistic argument. 

A further application was found by Bernshteyn and Kostochka~\cite{berkos} in characterizing
the $k$-list-critical graphs other than $K_k$ with a minimum number of edges. In an earlier work,
Bernshteyn, Kostochka and Pron~\cite{berkospron} characterized the graphs $G$ which are $C$-colorable
for all correspondence assignments $(L,C)$ such that $|L(v)|\ge \deg(v)$ for all $v\in V(G)$. Building on this,
they found the minimum number of edges in $k$-correspondence-critical graphs other than $K_k$, proving
it is the same bound as for $k$-list-critical graphs. Moreover, they characterized the graphs
which attain said bound, in turn characterizing the $k$-list-critical graphs which also attain
the bound, answering an open question from 2002. Their proof is inductive and uses the power of correspondence assignments.

Nevertheless, correspondence chromatic number seems to have many interesting properties on its own.
Many of the results for list coloring (especially those
based primarily on degeneracy arguments) translate to the setting of correspondence coloring; e.g., it is easy
to see that all planar graphs have correspondence chromatic number at most $5$ and that all planar graphs without cycles
of lengths $3$ and $4$ have correspondence chromatic number at most $3$, by mimicking
Thomassen's proofs~\cite{thomassen1994,thomassen1995-34} of the analogous choosability results.

Inspired by the previous versions of this paper, Bernshteyn~\cite{corr1} proved
that graphs of average degree $d$ have correspondence chromatic number $\Omega(d/\log d)$,
in contrast with the lower bound $\Omega(\log d)$ for their choosability~\cite{alondeg}.
Also, he showed that triangle-free graphs of maximum degree $\Delta$ have correspondence
chromatic number $O(\Delta/\log\Delta)$, thus determining the chromatic number of
regular triangle-free graphs up to a constant multiplicative factor.

In addition, correspondence chromatic number is related to lifts of graph.
A \emph{$k$-lift} of a graph $G$ is obtained from $G$ 
by replacing each vertex of $G$ with an independent set of size $k$ (called the \emph{fiber} of the replaced vertex) and by replacing each edge of $G$ with a perfect matching between the independent sets corresponding to its ends.
Thus if $H$ is a $k$-lift of $G$, then $V(H)=V(G)\times[k]$, and if $uv\in E(G)$, then $E(H[\{u,v\}\times[k]])$ is a perfect matching from $\{u\}\times[k]$ to $\{v\} \times [k]$. Clearly then
a $k$-lift is a $k$-correspondence assignment and indeed a $k$-correspondence assignment that is full on every edge is a $k$-lift.
Moreover, a $C$-coloring of such a full correspondence assignment is equivalent to an independent set of the $k$-lift
that intersects every fiber in exactly one vertex. Equivalently, if we replace every fiber by a clique of size $k$,
then a $C$-coloring is equivalent to an independent set of order $|V(G)|$ in the resulting graph.

Given this natural relationship to lifts, one may be tempted to think that
correspondence coloring is related to the unique games conjecture. But their
common relation to lifts is where that similarity ends. For in the unique games
conjecture, the desired subset of the lift is not independent, rather the
coloring sought is such that the colors of two adjacent vertices correspond if
possible and the value of said game is the maximum number of correspondences
that can be satisfied. In that sense, it is the very opposite of correspondence
coloring. 

One could also be tempted to generalize correspondence coloring by allowing any
subgraph instead of a matching for a correspondence. Such is certainly a
natural notion but of limited value. The known coloring proofs alluded to
before, such as Thomassen's, heavily use the fact that each color corresponds to
at most one other and hence only work in the setting of correspondence
coloring. Indeed, it is that property which drives their proofs and allows the
result to also hold for correspondence coloring. On the other hand, one could allow
correspondences to take the form of graphs of bounded degree, and in this setting nice theorems may still be
possible. However, such conditions are easily modeled using correspondence
coloring by allowing multiple edges. Indeed, Bernshteyn, Kostochka and Pron~\cite{berkospron} 
did precisely this in their work on correspondence assignments where the list size of each vertex is at least its degree.

\section{Straightness of correspondence assignments}

Before we start our work on the proof of Theorem~\ref{thm-maingen1}, let us introduce one more
idea that greatly simplifies some of the arguments, decreasing the number of $k$-corres\-pon\-dence assignments we need to consider.
We say that an edge $uv\in E(G)$ is \emph{straight} in a $k$-correspondence assignment $C$ for $G$ if every $(u,c_1)(v,c_2)\in E(C_{uv})$
satisfies $c_1=c_2$.  An edge $uv\in E(G)$ is \emph{full} if $C_{uv}$ is a perfect matching.
\begin{lemma}\label{lemma-tree}
Let $G$ be a graph with a $k$-correspondence assignment $C$.  Let $H$ be a subgraph of $G$ such that for every cycle $K$ in $H$,
the assignment $C$ is consistent on $K$ and all edges of $K$ are full.  Then there exists a $k$-correspondence assignment $C'$ for $G$
equivalent to $C$ such that all edges of $H$ are straight in $C'$, and $C'$ can be obtained from $C$ by renaming on vertices of $H$.
\end{lemma}
\begin{proof}
Without loss of generality, we can assume that $H$ is connected, as otherwise we perform the renaming on the components of $H$
separately.  Let $T$ be a spanning tree of $H$ rooted in a vertex $v$.  We perform renaming on the vertices of $T$ in DFS order
(with no renaming performed at $v$).  When we are processing a vertex $u$ whose parent in $T$ is $w$, we rename the colors at $u$
arbitrarily so that the edge $wu$ becomes straight.

Let $C'$ be the resulting $k$-correspondence assignment.  By the construction, all edges of $T$ are straight in $C'$.
Let us consider an edge $e\in E(H)\setminus E(T)$, and let $K$ be the unique cycle in $T+e$.  By assumption,
all the edges of $K$ are full, and all but $e$ are straight in $C'$.  By Observation~\ref{obs-consist}, $C'$ is consistent
on $K$, and thus $e$ is straight in $C'$ as well.
\end{proof}
Let us remark that in particular Lemma~\ref{lemma-tree} applies whenever $H$ is a forest.  Also, we cannot omit the assumption
that the edges of every cycle in $H$ are full---consider e.g. a triangle $K=v_1v_2v_3$ with a $2$-correspondence assignment $C$
with edges $(v_1,1)(v_2,1)$, $(v_2,2)(v_3,2)$, and $(v_3,1)(v_1,2)$, see Figure~\ref{fig-examples}(c).
Clearly, $C$ is consistent on $K$, but there exists no equivalent assignment such that all edges of $K$ are straight.

\section{The main result}

We are going to prove the following statement, which generalizes Theorem~\ref{thm-maingen1} by further allowing some of the
vertices to be precolored.

\begin{theorem}\label{thm-maingen}
Let $G$ be a plane graph without cycles of lengths $4$ to $8$.  Let $S$ be a set of vertices of $G$ such that either $|S|\le 1$, 
or $S$ consists of all vertices incident with a face of $G$.
Let $C$ be a $3$-correspondence assignment for $G$ such that $C$ is consistent on every closed walk of length $3$ in $G$.
If $|S|\le 12$, then for any $C$-coloring $\varphi_0$ of $G[S]$, there exists a $C$-coloring $\varphi$ of $G$ whose restriction
to $S$ is equal to $\varphi_0$.
\end{theorem}

Let us remark that in Theorem~\ref{thm-maingen}, we can without loss of generality assume that the vertices of $S$ are incident
with the outer face of $G$.
Let $B=(G,S,C,\varphi_0)$, where $G$ is a plane graph without cycles of lengths $4$ to $8$, $S\subseteq V(G)$ consists either of at most one
vertex incident with the outer face of $G$ or of all vertices incident with the outer face of $G$, $C$ is
a $3$-correspondence assignment for $G$ such that $C$ is consistent on every closed walk of length three in $G$,
and $\varphi_0$ is a $C$-coloring of $G[S]$.  If $|S|\le 12$, we say that $B$ is a \emph{target}.
Let $e(B)=|E(G)|-|E(G[S])|$ and let $s(B)=(|V(G)|,e(B), -\sum_{uv\in E(G)} |E(C_{uv})|)$.

We say that a target $B$ is a \emph{counterexample} if there exists no $C$-coloring $\varphi$ of $G$ whose restriction
to $S$ is equal to $\varphi_0$.  
We say that a counterexample $B$ is a \emph{minimal counterexample} if $s(B)$ is lexicographically minimum among all counterexamples;
that is, $|V(G)|$ is minimized, and subject to that, the number of edges of $G$ that do not join the vertices of $S$ is minimized, and
subject to those conditions the total number of edges in the matchings of the $3$-correspondence assignment $C$ is \emph{maximized}.

In order to prove Theorem~\ref{thm-maingen}, we aim to show that no minimal counterexamples exist.  Let us start by
establishing some basic properties of hypothetical minimal counterexamples.
We need a few more definitions.  Given a cycle $K$ in a graph, an edge $e$ is a \emph{chord} of $K$ if both ends of $e$ belong to $V(K)$, but $e\not\in E(K)$.
Let $|K|$ denote the length of the cycle $K$ (the number of its edges).  If $G$ is a $2$-connected plane graph, then every face $f$
is bounded by a cycle; let $|f|$ denote the length of $f$ (which in this case equals the length of the boundary cycle of $f$) and
let $V(f)$ denote the set of vertices incident with $f$.  By an \emph{open disk}, we mean a subset of the plane homeomorphic to $\{(x,y):x^2+y^2<1\}$.

\begin{lemma}\label{lemma-basic}
If $B=(G,S,C,\varphi_0)$ is a minimal counterexample, then
\begin{itemize}
\item[\textrm{(a)}] $V(G)\neq S$,
\item[\textrm{(b)}] $G$ is $2$-connected,
\item[\textrm{(c)}] for any cycle $K$ in $G$ that does not bound the outer face, if $K$ has length at most $12$, then the open disk bounded by $K$
does not contain any vertex,
\item[\textrm{(d)}] if $e_1$ and $e_2$ are distinct chords of a cycle $K$ in $G$ of length at most $12$, then there does not exist a triangle containing both $e_1$ and $e_2$,
\item[\textrm{(e)}] all vertices of $G$ of degree at most $2$ are contained in $S$,
\item[\textrm{(f)}] the outer face $F$ of $G$ is bounded by an induced cycle and $S=V(F)$, and
\item[\textrm{(g)}] if $P$ is a path in $G$ of length $2$ or $3$ with both ends in $S$ and no internal vertex in $S$, then
no edge of $P$ is contained in a triangle that intersects $S$ in at most one vertex.
\end{itemize}
\end{lemma}
\begin{proof}
Since $\varphi_0$ does not extend to a $C$-coloring of $G$, (a) obviously holds.

If $G$ were not connected, then each component
would have a $C$-coloring extending the restriction of $\varphi_0$ to the component by the minimality of $B$, and thus $G$ would
have a $C$-coloring extending $\varphi_0$.  This contradiction shows that $G$ is connected.
Suppose that $G$ is not $2$-connected; then $G=G_1\cup G_2$ for proper induced subgraphs $G_1$ and $G_2$ intersecting in one vertex $v$.
If $v\in S$, then note that both $G_1$ and $G_2$ have a $C$-coloring extending $\varphi_0$ and that the $C$-colorings
match on $v$, and thus $G$ has a $C$-coloring extending $\varphi_0$, a contradiction.  If $v\not\in S$, then without loss of generality, $V(G_2)\cap S=\emptyset$.
By the minimality of $B$, there exists a $C$-coloring $\varphi_1$ of $G_1$ that extends $\varphi_0$.  Let $C'$ be the restriction of $C$ to $G_2$,
let $S'=\{v\}$, and let $\varphi_0'$ be the coloring of $S'$ such that $\varphi'_0(v)=\varphi_1(v)$.  By the minimality of $B$, the target
$(G_2,S',C',\varphi'_0)$ is not a counterexample, and thus $G_2$ has a $C'$-coloring $\varphi_2$ that extends $\varphi'_0$.
Note that $\varphi_1$ and $\varphi_2$ together give a $C$-coloring of $G$ that extends $\varphi_0$, a contradiction.  Thus in both cases we obtain a contradiction,
showing that (b) holds.

Suppose that $K$ is a cycle of length at most $12$ in $G$ that does not bound the outer face of $G$, and that the open disk $\Lambda$ bounded by $K$
contains a vertex of $G$.  Let $G_1$ be the subgraph of $G$ induced by the vertices drawn in the complement of $\Lambda$,
and let $G_2$ consist of the vertices and edges of $G$ drawn in the closure of $\Lambda$.  Note that $S\subseteq V(G_1)$.
By the minimality of $B$, $\varphi_0$ extends to a $C$-coloring $\varphi_1$ of $G_1$, and the restriction of $\varphi_1$ to $K$ extends
to a $C$-coloring $\varphi_2$ of $G_2$.  Note that $\varphi_1$ and $\varphi_2$ together give a $C$-coloring of $G$ that extends $\varphi_0$.
Thus we obtain a contradiction, showing that (c) holds.

Suppose that a cycle $K=v_1v_2\ldots v_t$ of length $t\le 12$ has a chord, which by symmetry we can assume to be $v_1v_j$ for
some integer $j$ such that $3\le j\le t/2+1$.  Since $G$ does not contain cycles of lengths $4$ to $8$, we conclude that $j=3$,
and thus the chord $v_1v_3$ is contained in the triangle $v_1v_2v_3$.  However, since $G$ does not contain $4$-cycles, it follows
that $v_1v_3$ is not contained in any other triangle, showing that (d) holds.

If $v\in V(G)\setminus S$ had degree at most two, then we can extend $\varphi_0$ to a $C$-coloring $\varphi$ of $G-v$ by the minimality
of $B$, and furthermore we can select a color $\varphi(v)$ for $v$ such that for each neighbor $u$ of $v$, we have
$(u,\varphi(u))(v,\varphi(v))\not\in E(C_{uv})$.  This gives a $C$-coloring of $G$ that extends $\varphi_0$, which is a contradiction.
Hence, (e) holds.

Suppose now that $|S|\le 1$.  If $S=\emptyset$, then we can include an arbitrary vertex of the outer face of $G$ in $S$, and thus
without loss of generality, we can assume that $S=\{v\}$ for some vertex $v\in V(G)$.  If $v$ is contained in a cycle of length at most
$12$ in $G$, then by (c) $v$ is incident with a face $f$ of length at most $12$ whose boundary cycle is induced in $G$.  We can redraw $G$ so that $f$ is the outer face,
choose a $C$-coloring $\varphi'_0$ of the boundary of $f$, and conclude that $B'=(G,V(f),C,\varphi'_0)$ is a counterexample contradicting the
minimality of $G$, since $e(B')<e(B)$.  Hence, all cycles containing $v$ have length at least $13$.  Let $x$ and $y$ be the neighbors of $v$ in the
outer face, let $G+xy$ be drawn so that its outer face is $vxy$, let $C'$ be obtained from $C$ by letting $C'_{xy}$ be an edgeless
graph, and let $\varphi'_0$ be any $C'$-coloring of $vxy$.  Then $B''=(G+xy,\{v,x,y\},C',\varphi'_0)$ contradicts the minimality of $B$,
since $e(B'')<e(B)$.

Hence, it follows that $S$ consists of the vertices of the outer face of $G$, which is bounded by a cycle $F$.  If $F$ were not induced,
then (c) would imply that $V(G)=V(F)=S$, contradicting (a).  Hence, (f) holds.

Finally, suppose that $P$ is a path in $G$ of length $p\in \{2,3\}$ with both ends in $S$ and no internal vertex in $S$, and that an edge of
$P$ is contained in a triangle $T$ that intersects $S$ in at most one vertex.  Let $K_1$ and $K_2$ be the cycles of $F\cup P$ distinct from $F$.
Suppose that say $K_1$ is a triangle.  Since $P$ is a path, and thus its ends are distinct, this implies that $P$ has length $2$ and $K_1$ intersects
$S$ in two vertices.  This is a contradiction, since $G$ does not contain $4$-cycles and an edge of $P$ is contained in a different triangle $T$ (that intersects $S$
in at most one vertex).  Hence, neither $K_1$ nor $K_2$ is a triangle, and since $G$ does not contain cycles of lengths $4$ to $8$, it
follows that $|K_1|, |K_2|\ge 9$.  On the other hand, $|K_1|+|K_2|=|F|+2p\le 18$, and thus $|K_1|=|K_2|=9$.  By (c), we have $V(G)=S\cup V(P)$.
However, then an edge of $T$ is a chord of $K_1$ or $K_2$, and this edge together with a subpath of $K_1$ or $K_2$ forms a cycle of length $6$, $7$ or $8$,
which is a contradiction.  Hence, (g) holds.
\end{proof}

Next, let us simplify the correspondence assignment by adding new correspondences if doing so does not violate the assumptions.

\begin{lemma}\label{lemma-exedm}
Let $B=(G,S,C,\varphi_0)$ be a minimal counterexample.  If $e=uv$ is an edge of $G$ that does not join two vertices of $S$,
then $|E(C_{uv})|\ge 2$.  If additionally $e$ is not contained in a triangle, then $e$ is full in the assignment $C$.
\end{lemma}
\begin{proof}
If $e$ is not contained in a triangle and $e$ is not full, then let $c_1$ and $c_2$ be colors such that $(u,c_1)$ and $(v,c_2)$
are isolated vertices of $C_{uv}$.  Let $C'$ be the $3$-correspondence assignment obtained from $C$ by adding edge $(u,c_1)(v,c_2)$
to $C_{uv}$.  Then $(G,S,C',\varphi_0)$ is a counterexample contradicting the minimality of $B$.

Suppose now that $e$ is an edge of a triangle $T=uvw$ (the triangle is unique, since $G$ does not contain $4$-cycles),
and that $|E(C_{uv})|\le 1$.
If $E(C_{uv})=\emptyset$, then $(G-e,S,C,\varphi_0)$ is a counterexample contradicting the
minimality of $B$.  Otherwise, by Observation~\ref{obs-equiv} and Lemma~\ref{lemma-tree}, we can assume that
$C_{uv}$ contains only the edge $(u,1)(v,1)$.  For $a,b\in \{2,3\}$, let $C^{a,b}$ denote the $3$-correspondence assignment obtained from $C$
by adding edge $(u,a)(v,b)$ to $C_{uv}$, and let $B^{a,b}=(G,S,C^{a,b},\varphi_0)$.  Note that $s(B^{a,b})$ is lexicographically smaller than $s(B)$,
since the total number of edges in the matchings of the $3$-correspondence assignment $C^{a,b}$ is larger than the total number of edges in the matchings of $C$.
Since $G$ has no $C$-coloring, it also has no $C^{a,b}$-coloring.  However, the minimality of $B$ implies that $B^{a,b}$ is not a counterexample,
and thus $B^{a,b}$ is not a target; i.e., $T$ contains a closed walk that is not consistent in $C^{a,b}$.

Suppose that $(u,2)$ is isolated in $C_{uw}$.  Then, for $b\in \{2,3\}$, the only closed walk of length $3$
on which $C^{2,b}$ may be inconsistent is $uvwu$, and thus $C_{vw}\cup C_{wu}$ contains a path $(v,b)(w,d_b)(u,c_b)$ for some colors $d_b$ and $c_b\neq 2$.
Note that $c_2\neq c_3$, and by symmetry we can assume $c_2=1$.  But then the path $(v,1)(u,1)(w,d_2)(v,2)$ in $C_{vu}\cup C_{uw}\cup C_{wv}$
shows that $C$ is inconsistent on $vuwv$, contrary to the assumption that $C$ is consistent on closed walks of length $3$.

Hence, $(u,2)$ is not isolated in $C_{uw}$, and by symmetry, $(u,3)$ is not isolated in $C_{uw}$,
and $(v,2)$ and $(v,3)$ are not isolated in $C_{vw}$.  By the pigeonhole principle, there exist colors $c_u,c_v\in \{2,3\}$ and $c_w$
such that $C_{uw}\cup C_{wv}$ contains a path $(u,c_u)(w,c_w)(v,c_v)$. In that case $C^{c_u,c_v}$ is consistent on all closed walks
of length $3$, which is a contradiction.
\end{proof}

We can strengthen the previous result in the case of triangles that contain two vertices of degree three,
which appear in the main reducible configuration of our proof.

\begin{lemma}\label{lemma-exed}
Let $B=(G,S,C,\varphi_0)$ be a minimal counterexample, and let $T$ be a triangle in $G$
that has at least two vertices of degree three not belonging to $S$.
Then all the edges of $T$ are full in the assignment $C$.
\end{lemma}
\begin{proof}
Suppose for a contradiction that not all edges of $T$ are full.
Let $T=v_1v_2v_3$, where $v_1,v_2\not\in S$ have degree three.
For $i=1,2$, let $x_i$ be the neighbor of $v_i$ outside of the triangle.
By Observation~\ref{obs-equiv} and Lemma~\ref{lemma-tree}, we can assume that the edges $x_1v_1$, $x_2v_2$, $v_1v_2$ and
$v_1v_3$ are straight in $C$. By Lemma~\ref{lemma-exedm} and the pigeonhole principle, there exists a color $c$ such that $(v_1,c)$ is isolated neither
in $C_{v_1v_2}$ nor in $C_{v_1v_3}$; by symmetry, we can assume that $c=1$, and since the edges are straight,
$(v_1,1)(v_2,1)\in E(C_{v_1v_2})$ and $(v_1,1)(v_3,1)\in E(C_{v_1v_3})$.

Since $C$ is consistent on walks $v_2v_3v_1v_2$ and $v_3v_2v_1v_3$, either both $(v_2,1)$ and $(v_3,1)$ are
isolated in $C_{v_2v_3}$, or $(v_2,1)(v_3,1)\in E(C_{v_2v_3})$.
In the former case, letting $C'$ be the $3$-correspondence assignment obtained from $C$ by adding the edge $(v_2,1)(v_3,1)$
to $C_{v_2v_3}$, we find that $(G,S,C',\varphi_0)$ contradicts the minimality of $B$; hence, we may assume the latter case holds.

Let $D$ be the $3$-correspondence assignment for $G$ that matches $C$ on $E(G)\setminus E(T)$ and has the property that all edges of $T$
are straight and full in $D$.  By the minimality of $B$, we conclude that there exists a $D$-coloring $\varphi'$ of $G$ that extends $\varphi_0$,
and since $B$ is a counterexample, $\varphi'$ is not a $C$-coloring of $G$.  Since $D$ differs from $C$ only on the edges of $T$ and all edges of $T$ other than $v_2v_3$ are straight in $C$,
by symmetry we can assume that $\varphi'(v_2)=2$, $\varphi'(v_3)=3$ and $(v_2,2)(v_3,3)\in E(C_{v_2v_3})$.
If $(v_3,3)$ is isolated in $C_{v_3v_1}$, then we can modify $\varphi'$ to a $C$-coloring of $G$ by recoloring $v_2$ by a color $c$
in $\{1,3\}$ different from $\varphi'(x_2)$
and by recoloring $v_1$ by a color different from $c$ and $\varphi'(x_1)$.

Hence, we can assume that $(v_3,3)$ is
not isolated in $E(C_{v_3v_1})$, and since the edge $v_1v_3$ is straight in $C$, we have $(v_3,3)(v_1,3)\in E(C_{v_3v_1})$.
But, by Lemma~\ref{lemma-exedm}, we have either $(v_1,2)(v_2,2)\in E(C_{v_1v_2})$ or $(v_1,3)(v_2,3)\in E(C_{v_1v_2})$.
In either case, $C$ is not consistent on all closed walks of length $3$ in $T$, which is a contradiction.
\end{proof}

\begin{figure}
\begin{center}
\includegraphics{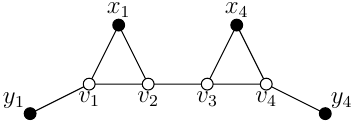}
\end{center}
\caption{A tetrad.}\label{fig-tetrad}
\end{figure}

Finally, we are ready to deal with the main reduction of the proof (inspired by~\cite{bor47}).
A \emph{tetrad} in a plane graph is a path $v_1v_2v_3v_4$ of vertices of degree three contained in the boundary of a
face, such that both $v_1v_2$ and $v_3v_4$ are edges of triangles, see Figure~\ref{fig-tetrad}.

\begin{lemma}\label{lemma-tetrad}
If $B=(G,S,C,\varphi_0)$ is a minimal counterexample, then every tetrad in $G$ contains a vertex of $S$.
\end{lemma}
\begin{proof}
Suppose for a contradiction that $v_1v_2v_3v_4$ is a tetrad in $G$ disjoint from $S$.  Let $v_1v_2x_1$ and $v_3v_4x_4$ be triangles,
and for $i=1,4$, let $y_i\neq x_i$ be the other neighbor of $v_i$ not in the tetrad.  Since $G$ has no cycles of lengths $4$ to $8$, all these
vertices are pairwise distinct.

Considering the path $x_1v_2v_3x_4$, Lemma~\ref{lemma-basic}(g) implies that at most one of $x_1$ and $x_4$ belongs to $S$.
If say $x_1\in S$, then $x_4\not\in S$ and Lemma~\ref{lemma-basic}(g) applied to the path $y_1v_1x_1$ implies that $y_1\not\in S$.
If $x_1,x_4\not\in S$ and $y_1,y_4\in S$, then Lemma~\ref{lemma-basic}(g) implies that neighbors of $x_1$ and $x_4$ do not belong to $S$.
Hence, we can by symmetry assume that $x_4\not\in S$ and either $y_1\not\in S$, or $x_4$ has no neighbor in $S$.

Note that by Lemma~\ref{lemma-basic}(c) and (d), $G-\{v_1,v_2,v_3,v_4\}$ does not contain any path of length at most $8$
between $y_1$ and $x_4$ (such a path together with the path $y_1v_1v_2v_3x_4$ would form a cycle $K$ of length at most $12$
with two of the edges of one of the triangles $v_1v_2x_1$ or $v_3x_4v_4$ contained in the open disk bounded by $K$).
By Lemmas~\ref{lemma-exedm} and \ref{lemma-exed}, all edges of $G$ incident with $\{v_1, \ldots, v_4\}$ are full.
By Observation~\ref{obs-equiv} and Lemma~\ref{lemma-tree}, we can also assume that
the edges incident with $\{v_1, \ldots, v_4\}$ are straight.  Let $G'$ be obtained from $G-\{v_1,\ldots,v_4\}$ by identifying
$y_1$ with $x_4$, and let $C'$ be the restriction of $C$ to $E(G')$.  Note that the identification does not create an edge between
vertices of $S$, and thus $\varphi_0$ is a $C'$-coloring of the subgraph of $G'$ induced by $S$.  Also, the identification does not create any cycle of length
at most $8$, and thus $G'$ contains no cycles of lengths $4$ to $8$ and $C'$ is consistent on all closed walks of length three in $G'$.

By the minimality of $B$, we conclude that $G'$ has a $C'$-coloring that extends $\varphi_0$.
Consequently, $G-\{v_1,\ldots,v_4\}$ has a $C'$-coloring $\varphi$ extending $\varphi_0$ such that $\varphi(y_1) = \varphi(x_4)$.
We can extend $\varphi$ to a $C$-coloring of $G$,
by coloring $v_4$ and $v_3$ in order by colors distinct from the colors of their neighbors, and then choosing colors for
$v_1$ and $v_2$; the last part is possible, since all edges incident with $v_1$ and $v_2$ are straight and $y_1$ and $v_3$
have different colors.  This is a contradiction.
\end{proof}

We say that a vertex $v$ of a target $B=(G,S,C,\varphi_0)$ is \emph{light} if $\deg(v)=3$, $v$ is incident with a triangle,
and $v\not\in S$.  Lemma~\ref{lemma-tetrad} limits the number of light vertices incident with each face as follows.

\begin{corollary}\label{cor-consec}
If $B=(G,S,C,\varphi_0)$ is a minimal counterexample, then no face of $G$ is incident with five consecutive light vertices.
Furthermore, if a face of $G$ is incident with consecutive vertices $v_0$, $v_1$, \ldots, $v_5$ and the vertices
$v_1, \ldots, v_4$ are light, then the edges $v_0v_1$, $v_2v_3$, and $v_4v_5$ are incident with triangles.
\end{corollary}
\begin{proof}
Suppose that a face $f$ is incident with consecutive light vertices $v_1,\ldots, v_5$.  Since $v_3$ is incident with a triangle and
$\deg(v_3)=3$, we can by symmetry assume that the edge $v_3v_4$ is incident with a triangle.
Since $G$ does not contain $4$-cycles and $\deg(v_3)=3$, the edge $v_2v_3$ is not incident with a triangle, and thus the triangle incident with $v_2$
contains the edge $v_1v_2$.  However, then $v_1v_2v_3v_4$ is a tetrad contradicting Lemma~\ref{lemma-tetrad}.

Similarly, if a face is incident with consecutive vertices $v_0$, $v_1$, \ldots, $v_5$ and the vertices $v_1, \ldots, v_4$ are light, then
either edges $\{v_1v_2,v_3v_4\}$, or $\{v_0v_1,v_2v_3,v_4v_5\}$ are incident with triangles, and the former is not possible
by Lemma~\ref{lemma-tetrad}.
\end{proof}

We can now conclude the proof by a quite straightforward discharging argument.

\begin{proof}[Proof of Theorem~\ref{thm-maingen}]
Suppose for a contradiction that Theorem~\ref{thm-maingen} is false, and thus there exists a target $B=(G,S,C,\varphi_0)$ such that
$\varphi_0$ does not extend to a $C$-coloring of $G$, i.e., $B$ is a counterexample.  Choose such a counterexamle $B$ with $s(B)$ lexicographically minimum,
so that $B$ is a minimal counterexample.

Let us now assign the initial charge $\ch_0(v)=2\deg(v)-6$ to each vertex $v\in V(G)$ and $\ch_0(f)=|f|-6$ to each face $f$ of $G$.
By Euler's formula, the sum of the charges is $-12$.  The charge is then redistributed according to the following rules:
\begin{itemize}
\item[(R1)] Each vertex incident with a non-outer face $f$ of length three sends $1$ to $f$.
\item[(R2)] Each face of length at least $9$ sends $1/2$ to each incident light vertex.
\item[(R3)] Let $v$ be a vertex of degree at least $4$ and let $f$ be a face of length at least $9$ incident with $v$.
If $\deg(v)\ge 5$, then $v$ sends $1/2$ to $f$.
If $\deg(v)=4$ and $v$ is incident with exactly one triangle $vxy$ and either the edge $vx$ or the edge $vy$ is incident with $f$,
then $v$ also sends $1/2$ to $f$.
\item[(R4)] If $v\in S$ has degree two, then the non-outer face incident with $v$ sends $1/2$ to $v$.
\end{itemize}
This does not affect the total amount of the charge, and thus the sum of the charges is still $-12$.
Let the resulting final charge be denoted by $\ch$.
Next, we discuss the final charge of vertices and faces of $G$.

Let us first consider a vertex $v\in V(G)$, and
let $t$ denote the number of triangles containing $v$.  Since $G$ does not contain $4$-cycles, we have $t\le \deg(v)/2$,
and thus $v$ sends at most $t+(\deg(v)-t)/2\le \frac{3}{4}\deg(v)$ by rules (R1) and (R3).  Hence, if $\deg(v)\ge 5$,
then $\ch(v)\ge \ch_0(v)-\frac{3}{4}\deg(v)=\frac{5}{4}\deg(v)-6>0$.
If $\deg(v)\in\{3,4\}$ and $t=0$, then $\ch(v)=\ch_0(v)\ge0$. 
If $\deg(v)=3$ and $t=1$, then by (R1) and (R2), we have $\ch(v)=0$ if $v\not\in S$ and $\ch(v)=-1$ if $v\in S$.
If $\deg(v)=4$ and $t=1$, then $\ch(v)=\ch_0(v)-1-2\times\frac{1}{2}=0$ by (R1) and (R3).
If $\deg(v)=4$ and $t=2$, then $\ch(v)=\ch_0(v)-2\times 1=0$ by (R1).  Finally, if $\deg(v)=2$, then by
Lemma~\ref{lemma-basic}(e) we have $v\in S$, and $\ch(v)=-3/2$ by (R4).
Hence, 
\begin{itemize}
\item[($\dagger$)] \emph{each vertex $v\in V(G)$ satisfies $\ch(v)\ge 0$ if $v\not\in S$ or if $v$ has degree at least four,
$\ch(v)\ge -1$ if $v\in S$ and $v$ has degree three, and $\ch(v)\ge -3/2$ if $v\in S$ and $v$ has degree two.}
\end{itemize}

Furthermore, we claim that
\begin{itemize}
\item[($\star$)] \emph{every face $f$ of $G$ distinct from the outer face $F$ has non-negative final charge.}
\end{itemize}
\begin{subproof}
If $|f|=3$, then $\ch(f)=\ch_0(f)+3=0$ by (R1).
So we may suppose that $|f|\ge 9$.  Let $R$ denote the set of light vertices incident with $f$, and
let $s$ denote the number of vertices of $f$ of degree two.  Note that $|R|+s\le |f|$.  Also, since both neighbors of
a vertex of degree two belong to $S$ and $G$ is $2$-connected, if $s\neq 0$, then $f$ is incident with at least two
vertices of $S$ of degree at least three, and $|R|+s\le |f|-2$.  By (R2) and (R4),
$\ch(f)\ge \ch_0(f)-(|R|+s)/2\ge |f|/2-6$.  Hence, $\ch(f)\ge 0$ unless $|f|=11=|R|+s$, or $|f|=10$ and $|R|+s\ge 9$,
or $|f|=9$ and $|R|+s\ge 7$.

\begin{figure}
\begin{center}
\includegraphics{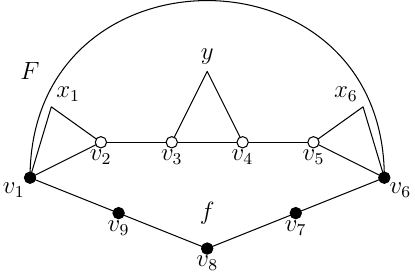}
\end{center}
\caption{The last case in the proof of Theorem~\ref{thm-maingen}.}\label{fig-lastcase}
\end{figure}

If $|f|=11$ and $|R|+s=11>|f|-2$, or $|f|=10$ and $|R|+s\ge 9>|f|-2$, then $s=0$ and at most one vertex incident with $f$ does not belong to $R$;
this contradicts Corollary~\ref{cor-consec}.  Therefore, $|f|=9$ and $|R|+s\ge 7$.  Let us discuss two subcases.
\begin{itemize}
\item Suppose $s=0$. Then $|R|\ge 7$.  By Corollary~\ref{cor-consec}, no 5 consecutive vertices of $f$ belong to $R$.
Hence, we may assume that $|R|=7$ and we can label the vertices of $f$ by $v_1$, \ldots, $v_9$ in order so that $v_1,v_6\not\in R$ and $v_2,\ldots, v_5,v_7,\ldots,v_9\in R$.
By Corollary~\ref{cor-consec}, the edges $v_1v_2$, $v_3v_4$ and $v_5v_6$ are incident with triangles, and by symmetry,
we can assume that $v_7v_8$ is incident with a triangle.  Since $v_7$ has degree three, the edge $v_6v_7$ is not
incident with a triangle.  Since $v_5,v_7\not\in S$, we conclude that $v_6$ cannot both have degree three and belong to $S$.
Since $v_6$ is not light, it follows that $v_6$ has degree at least $4$ and it sends $1/2$ to $f$ by (R3).
Hence, $\ch(f)=\ch_0(f)-|R|/2+1/2=0$.
\item Suppose $s>0$. Then we have $|R|+s=7$.  Consequently, $f$ is incident with exactly two vertices of $S$ of degree at least three,
and thus the vertices of degree $2$ are consecutive in the boundary of $f$.  If $s=7$, then $V(f)\subseteq V(F)$;
but then either $V(G)=S$ or $F$ has a chord, which contradicts Lemma~\ref{lemma-basic}(a) and (f).  Hence, $s<7$
and $R\neq\emptyset$.  Let $P$ be the path obtained from the boundary cycle of $f$ by
removing the vertices of degree two (all the internal vertices of $P$ belong to $R$), and let $K$ be the cycle in $F\cup P$
distinct from $F$ and $f$.  Note that $|K|+|f|=|F|+2|R|+2$, and thus $|K|=|F|+2|R|-7$.

If $|R|\le 3$, then $|K|<|F|\le 12$, and by Lemma~\ref{lemma-basic}(c), we have $V(G)=V(F)\cup R$.  Since all vertices of $R$
have degree three, each of them is incident with a chord of $K$.  Consequently, $K$ is not a triangle. If $|R|\le 2$, then $|K|\le |F|-3\le 9$,
and since $K$ has a chord, we conclude that $G$ contains a cycle of length $4$ to $8$, which is a contradiction.  If $|R|=3$, then
the middle vertex of $P$ is incident with a triangle whose two edges are chords of $K$, which contradicts Lemma~\ref{lemma-basic}(d).

Hence, we may assume that $|R|\ge 4$, and by Corollary~\ref{cor-consec} we have $|R|=4$.  Let us label the vertices of $f$ by $v_1$, \ldots $v_9$
in order so that $R=\{v_2,v_3,v_4,v_5\}$.  Corollary~\ref{cor-consec} implies that $v_1v_2$, $v_3v_4$ and $v_5v_6$ are incident with
triangles $v_1v_2x_1$, $v_3v_4y$ and $v_5v_6x_6$, see Figure~\ref{fig-lastcase}.  If $v_1$ or $v_6$ has degree at least $4$, then it sends $1/2$ to $f$ by (R3),
and $\ch(f)\ge \ch_0(f)-(|R|+s)/2+1/2=0$.  Otherwise, $x_1,x_6\in V(F)$.  Consider the cycle $K_1$ in $F\cup x_1v_2v_3v_4v_5x_6$ such that the open disk bounded by
$K_1$ does not contain $f$.  We have $|K_1|=|F|-1\le 11$, and thus Lemma~\ref{lemma-basic}(c) implies that $V(G)=V(F)\cup R$.
Consequently, the edges $v_3y$ and $v_4y$ are chords of $K_1$, contradicting Lemma~\ref{lemma-basic}(d).
\end{itemize}
This finishes the proof of the claim $(\star)$ that all the faces of $G$ other than $F$ have non-negative final charge.
\end{subproof}

Note that the outer face $F$ has charge at least $|F|-6$.
Since $G$ is $2$-connected, at most $|F|-2$ vertices of $F$ have degree two.  By ($\dagger$) and $(\star)$, the sum of the final charges
of vertices and faces of $G$ is at least $|F|-6-\frac{3}{2}(|F|-2)-2=-5-|F|/2\ge -11$.  This is a contradiction, since by the choice of
the initial charge, the sum of the charges is $-12$.
\end{proof}

\bibliographystyle{siam}
\bibliography{lisstein}

\end{document}